%% file: main.tex
\documentclass[letterpaper, 10 pt, conference]{ieeeconf}  

\IEEEoverridecommandlockouts                              
\usepackage{graphicx} 
\usepackage{amsmath} 
\usepackage{amssymb}  
\usepackage{amstext,enumerate}
\usepackage{subfigure}
\usepackage{url}
\usepackage[symbol]{footmisc}

\usepackage{latexsym, color,booktabs}

\newtheorem{assum}{Assumption}
\newtheorem{prop}{Proposition}
\newtheorem{thm}{Theorem}
\newtheorem{rem}{Remark}
\newtheorem{lem}{Lemma}

\newtheorem{defn}{Definition}

\newtheorem{objective}{Objective}

\title{\LARGE \bf
Controller confidentiality for nonlinear systems under sensor attacks 
}

\author{Michelle S. Chong
\thanks{M. Chong is with the Control Systems Technology section at the Department of Mechanical Engineering, Eindhoven University of Technology. 
        {\tt\small m.s.t.chong@tue.nl} }

}

\begin{document}
\maketitle
\thispagestyle{empty}
\pagestyle{empty}

\begin{abstract}
Controller confidentiality under sensor attacks refers to whether the internal states of the controller can be estimated when the adversary knows the model of the plant and controller, while only having access to sensors, but not the actuators. We show that the controller's state can be estimated accurately when the nonlinear closed-loop system is detectable. In the absence of detectability, controller confidentiality can still be breached with a periodic probing scheme via the sensors under a robust observability assumption, which allows for the controller's state to be estimated with arbitrary accuracy during the probing period, and with bounded error during the non-probing period. Further, stealth can be maintained by choosing an appropriate probing duration. This study shows that the controller confidentiality for nonlinear systems can be breached by balancing the estimation precision and the stealthiness of the adversary.
\end{abstract}

\input{intro.tex}
\input{prelim.tex}
\input{problem.tex}

\input{result-estimation-detectable.tex}

\input{result-estimation-not.tex}
\input{conclusion.tex}

\appendix
\input{appendix.tex}

\bibliographystyle{ieeetr}
\bibliography{grid_model_analysis.bib}

\end{document}

%% file: intro.tex
\section{Introduction} \label{sec:intro}

The \textit{cyber} security of dynamical systems have gained traction in recent years as cyber-physical systems become increasingly interconnected, see \cite{sandberg2015cyberphysical} and \cite{chong2019tutorial} for a tutorial overview. While the connectivity improves performance and enhances the capabilities of cyber-physical systems, it also exposes vulnerabilities which can be exploited maliciously. The objective of the adversary is to gather data in order to launch an attack to disrupt operation, while avoiding detection by the system operator.

Although there are many vulnerable points in cyber-physical systems, the vulnerability of sensors has been widely studied thus far. In this setup, a subset of the sensor measurements can be read and manipulated by an adversary and various attack strategies have been investigated to avoid detection in works by \cite{cardenas2011attacks, murguia2019model, guo2018worst} to name a few, and to then still provide good estimates of the system states in works by \cite{an2017secure, shoukry2018smt,kim2018detection, chong2020secure} and more. Underlying the attack strategies mentioned earlier is the adversary's knowledge of the controller's state, which motivated a line of work investigating \textit{the confidentiality of control systems} \cite{umsonst2021confidentiality, dibaji2018secure, xue2014security, yuan2015security}. In all these works, control systems with only linear dynamics is considered.  

In this paper, we analyse the controller confidentiality of \textit{nonlinear systems}. Precisely, we provide rigorous analysis on whether the states of the controller can be estimated when the adversary can read and manipulate the sensors. We consider plant and controllers models with a general nonlinear structure, which already has some inherent stability properties, as all well designed control systems possess. The adversary knows the plant and controller models and has access to the sensors, but not the actuators. We show that if the closed-loop system is detectable (assumed in \cite{umsonst2021confidentiality}), then the adversary only needs to read the sensors and not manipulate them to reconstruct the controller's state exactly. In the absence of closed-loop detectability, the adversary needs to manipulate the sensors, which we call the act of probing, such that the controller's states can be estimated within a bounded margin of error. As the adversary now needs to probe the closed-loop system, this could raise alarms as anomaly detection schemes are typically employed in well designed control systems. In this scenario, we show that stealth can be maintained when the adversary employs a dual-mode probing scheme.

First, we assume that the closed-loop system is semiglobal asymptotically stable and has a robust observability property. With these assumptions, the estimation of the controller's state (with bounded error) and stealth (semiglobal practical stability of the closed-loop system) can be achieved. To do so, the adversary probes the closed-loop system via the sensors periodically for a short period. During which, a fast estimator can reconstruct the controller's state with desired precision. After which, the probing signal is turned off for a specified time interval to preserve the semiglobal practical stability of the closed loop system (maintain stealth), while still keeping the estimated controller's state within a neighborhood of the true controller's state. During the non-probing interval, the estimator is turned off and the estimate of the controller's state is held until the end of the non-probing interval. This scheme is reminiscent of the time-sharing strategies in \cite{shim2003asymptotic} and \cite{nesic2000output} for sampled-data output feedback for nonlinear systems and Wiener systems, respectively. This paper focuses only on confidentiality breaching strategies. Hence, future work will involve developing defense strategies which involve the introduction of uncertainties known to the system operator, but not known to the adversary, for example.  

Our paper is organised as follows. Preliminaries are introduced in Section \ref{sec:prelim} and the problem is motivated and formulated in Section \ref{sec:problem}. A non-invasive breach of controller confidentiality is analysed in Section \ref{sec:result_estimate_detectable}. Next, an invasive strategy is proposed in Section \ref{sec:result_estimate_not} where a dual-mode time shared probing strategy is proposed. The invasive strategy is shown to achieve the aim of stealthy estimation in Section \ref{sec:main_result}. We conclude the paper with Section \ref{sec:conclude} and proofs are provided in the Appendix.
 

%% file: prelim.tex
\section{Preliminaries} \label{sec:prelim}
 	 Let $\mathbb{R}=(-\infty,\infty)$, $\mathbb{R}_{\geq 0}=[0,\infty)$, $\mathbb{R}_{>0}=(0,\infty)$. 
 	 Let $\mathbb{N}_{\geq i}=\{i,i+1,i+2,\dots\}$. A finite set of integers $\{i,i+1,i+2,\dots,i+k\}$ is denoted as $\mathbb{N}_{[i,i+k]}$.
 	 The identity matrix of dimension $n$ is denoted by $\mathbb{I}_{n}$. 
	A diagonal matrix with matrices $d_i$, $i\in\mathbb{N}_{[1,n]}$ is denoted by $\textrm{diag}(d_1,d_2,\dots,d_n)$.
 	Given a symmetric matrix $P$, its maximum (minimum) eigenvalue is denoted by $\lambda_{\max}(P)$ $(\lambda_{\min}(P))$.
	The infinity norm of a vector $x \in \mathbb{R}^{n}$, is denoted $|x|:= \underset{i\in\mathbb{N}_{[1,n]}}{\max} \left| x_i \right|$ and for a matrix $A\in\mathbb{R}^{n\times n}$, $|A|:= \underset{i\in\mathbb{N}_{[1,n]}}{\max} \underset{j\in\mathbb{N}_{[1,n]}}{\sum}|a_{ij}|$, where $a_{ij}$ is the row $i$-th and column $j$-th element of matrix $A$. 
     
 	A continuous function $\alpha:\mathbb{R}_{\geq 0}\to\mathbb{R}_{\geq 0}$ is a class $\mathcal{K}$ function, if it is strictly increasing and $\alpha(0)=0$; additionally, if $\alpha(r)\to\infty$ as $r\to\infty$, then $\alpha$ is a class $\mathcal{K}_{\infty}$ function. A continuous function $\beta:\mathbb{R}_{\geq0}\times \mathbb{R}_{\geq 0} \to \mathbb{R}_{\geq 0}$ is a class $\mathcal{KL}$ function, if: (i) $\beta(.,s)$ is a class $\mathcal{K}$ function for each $s\geq 0$; (ii) $\beta(r,.)$ is non-increasing and (iii) $\beta(r,s)\to 0$ as $s\to \infty$ for each $r\geq 0$.

%% file: problem.tex
\section{Motivation and problem formulation} \label{sec:problem}

\subsection{Plant and controller models} \label{sec:plant_controller}
We consider nonlinear systems of the form
\begin{align}
 \Sigma_p:  \qquad   \dot{x}_p & =  f_p(x_p,u), \label{eq:system}\\
        y & =  h(x_p) + a, \label{eq:system_output}
\end{align} 
where $x_p\in\mathbb{R}^{n_p}$ is the system's state, $u\in\mathbb{R}^{n_u}$ is the input, $y\in\mathbb{R}^{n_y}$ is the output and $a:\mathbb{R}_{\geq 0}\to \mathbb{R}^{n_y}$ is an attack signal, respectively. The functions $f_p$ is locally Lipschitz and $h$ is sufficiently smooth. 

We consider controllers with a general nonlinear structure taking the following form
\begin{align}  
  \Sigma_c: \qquad  \dot{x}_{c} & =  f_c(x_c,y), \qquad u  =  \kappa(x_c,y), \label{eq:control}
\end{align}
where $x_c \in \mathbb{R}^{n_c}$ is the controller's state,  the locally Lipschitz function $\kappa:\mathbb{R}^{n_c}\to\mathbb{R}^{n_u}$ is the control law and the function $f_c:\mathbb{R}^{n_c}\times \mathbb{R}^{n_y} \to \mathbb{R}^{n_c}$ is locally Lipschitz such that for all initial conditions $x(0)\in\mathbb{R}^{n_p}$ and $x_c(0)\in\mathbb{R}^{n_c}$, the trajectories of \eqref{eq:system}, \eqref{eq:system_output} and \eqref{eq:control} exist for all time $t\geq 0$. 

The controller model in \eqref{eq:control} captures both state and output feedback schemes. In the case where state feedback is employed to stabilise the plant \eqref{eq:system}, the plant output \eqref{eq:system_output} is $h(x_p)=x_p$ (in the absence of sensor attacks) and the controller model in \eqref{eq:control} becomes $f_c(x_c,y)=0$, $x_c(0)=0$ and $\kappa(x_c,y)=\kappa(0,h(x_p))$. When an output feedback stabilisation scheme is used, then the controller model \eqref{eq:control} takes the role of a state observer of the plant \eqref{eq:system} with $\kappa(x_c,y)$ being the control law.

In this paper, we focus on control schemes \eqref{eq:control} which render the closed-loop system composed of \eqref{eq:system}, \eqref{eq:system_output} and \eqref{eq:control} semiglobally asymptotically stable in the absence of sensor attacks ($a(t)=0$, for all $t\geq 0$) as stated in the assumption below.
\begin{assum}[Closed-loop system is SG-AS] \label{assum:closed_loop_AS}
    Let $x:=(x_p,x_c)$. The closed loop system from \eqref{eq:system}, \eqref{eq:system_output} and \eqref{eq:control}  with the following dynamics  for all $t\geq 0$,
\begin{align} \label{eq:closed_loop}
    \dot{x} & = \left(\begin{array}{c} f_p(x_p,\kappa(x_c,h(x_p))) \\ f_c(x_c, h(x_p)) \end{array}\right) =: f(x, y),
\end{align}
with $a(t)=0$, is asymptotically stable, i.e., there exist a class $\mathcal{KL}$ function $\beta_{z}$ such that 
\begin{equation}
    |x(t)| \leq \beta_x(|x(0)|,t), \; \forall t\geq 0.  \label{eq:closeloop_AS}
\end{equation}
When \eqref{eq:closeloop_AS} holds for $|x(0)|\leq \Delta_x$, where $\Delta_x>0$, we say that the closed-loop system \eqref{eq:closed_loop} is semiglobal asymptotically stable (SG-AS).
\hfill $\Box$
\end{assum}

Control schemes \eqref{eq:control} for nonlinear systems which involve output feedback results in a closed-loop system that is SG-AS for certain classes of systems, see \cite{ khalil1993semi, teel1994global, atassi1999separation, yang2014semi, lin2019nonlinear}, for example. For linear plant and controllers, this problem is well studied and Assumption \ref{assum:closed_loop_AS} holds thanks to the well-known separation principle which yields a controller \eqref{eq:control} that results in a closed loop system that is globally exponentially stable.

\subsection{Adversary model and objectives} \label{sec:attack_objectives}
We assume that the adversary has knowledge of the plant and controller models, but not their initial conditions $x(0)$ and $x_c(0)$. The adversary can manipulate the sensors $y$, but does not have access to the actuators $u$. Precisely, the adversary operates under the following conditions.
\begin{assum}[Adversary model]  \ \label{assum:adversary} 
    \begin{enumerate}
        \item The adversary can manipulate the sensor readings $h(x_p)$ via an attack signal $a$, modelled by \eqref{eq:system_output}.
        \item The adversary knows the functions $f_p$, $h$, $f_c$ and $\kappa$ from \eqref{eq:system}, \eqref{eq:system_output}, \eqref{eq:control}.
        \item The adversary does not know the control input $u$, nor the initial state of the plant and controller models $x(0)$. 
    \end{enumerate} \hfill $\Box$
\end{assum}

The objectives of the adversary are to obtain an estimate of the controller's states $x_c$ under the operating conditions stated in Assumption \ref{assum:adversary} without letting the state of the closed loop system $x(t)$ become unbounded in finite time, in the sense that $\underset{t\to T}{\lim} |x(t)| = \infty$, for $T<\infty$. We state these two objectives precisely below.

\begin{objective}[Estimation of the controller's state $x_c$] \label{obj:estimate}
    The estimate of the controller's state \eqref{eq:control} denoted by $\hat{x}_c$ converges to a neighbourhood of the controller's state $x_c$. \hfill $\Box$
\end{objective}

\begin{objective}[Maintaining stealth] \label{obj:stealth}
    The closed loop system \eqref{eq:closed_loop} is semiglobal practically stable, i.e. for any $K_x \geq \Delta_x >0$ , the solution to the closed loop system \eqref{eq:closed_loop} satisfies
    \begin{equation} \label{eq:stealth_obj}
        |x(0)|\leq \Delta_x \implies |x(t)| \leq  K_x, \; \forall t\geq 0.
    \end{equation} \hfill $\Box$
\end{objective}

When both of the aforementioned objectives are achieved, we say that the adversary has achieved stealthy estimation of the controller's state. In other words, the confidentiality of the control system has been breached. Figure \ref{fig:setup} illustrates the problem setup
\begin{figure}[h!]
    \centering
    \includegraphics[scale=0.95]{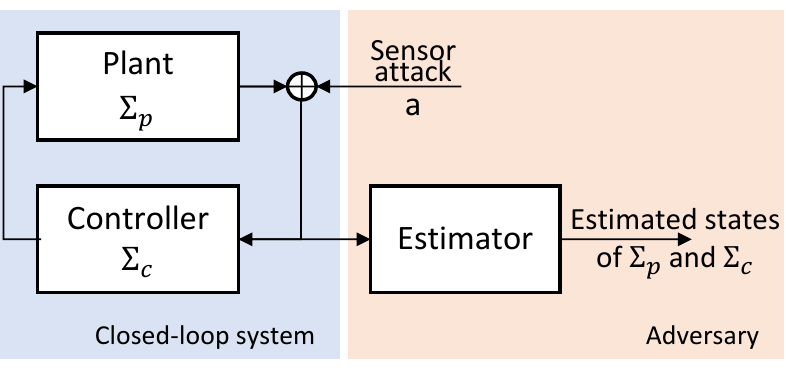} \vspace{-1em}
    \caption{Problem setup}
    \label{fig:setup}
\end{figure} 

In the sections that follow, we describe how an adversary can achieve these goals. The stealthy estimation of the controller's state $x_c$ can be realised without manipulating the sensor measurements when the closed-loop system \eqref{eq:closed_loop} is detectable in Section \ref{sec:result_estimate_detectable} and by using the sensor measurements $y$ to probe the closed-loop system in a time-shared manner in Section \ref{sec:result_estimate_not}.

%% file: result-estimation-detectable.tex
\section{Closed-loop system \eqref{eq:closed_loop} is detectable} \label{sec:result_estimate_detectable}

We first consider the case where the closed-loop system \eqref{eq:closed_loop} is detectable, which is defined as follows.

\begin{defn}[Detectability] \label{def:detectable}
    The closed-loop system \eqref{eq:closed_loop} is detectable if there exists a function $l:\mathbb{R}^{n_y}\to\mathbb{R}^{n_x+n_c}$ with $l(0)=0$, such that the estimate $\hat{x}:=(\hat{x}_p, \hat{x}_c)$ is the solution to the following system with dynamics given by
\begin{align} \label{eq:detectable_observer}
    \dot{\hat{x}} & = f(\hat{x},y) + l(y-\hat{y}), \qquad \hat{y} = h(\hat{x}_p),
\end{align}
and the closed-loop system \eqref{eq:closed_loop} satisfy
\begin{align} \label{eq:error_detectable_observer}
    |\hat{x}(t)-x(t)| \leq \beta_{\hat{x}}(|\hat{x}(0)-x(0)|,t)+\gamma_x\left(\underset{s\in[0,t]}{\sup}|a(s)|\right),
\end{align}
for all $t\geq 0$, for all initial conditions $\hat{x}(0)$, $x(0)\in\mathbb{R}^{n_x+n_c}$, $\beta_x \in \mathcal{KL}$ and $\gamma_x \in \mathcal{K}$. \hfill $\Box$
\end{defn}

The function $l(\hat{y}-y)$ is known as an output injection term and the system \eqref{eq:detectable_observer} whose solution $\hat{x}$ provides the estimate of $x$ is known in the literature as a nonlinear observer. According to Definition \ref{def:detectable}, observers \eqref{eq:detectable_observer} with property \eqref{eq:error_detectable_observer} are known as input-to-state (ISS) observers with respect to the attack signal $a$. The design of observers \eqref{eq:detectable_observer} for detectable systems \eqref{eq:closed_loop} according to Definition \ref{def:detectable} is done for specific classes of systems, which exploits the inherent structure of the system, see \cite{besanccon2007nonlinear} for an overview. The following linear time-invariant system 
\begin{equation} \label{eq:linear_sys}
    \dot{x} = A x + B y, \qquad y = \left[\begin{array}{cc}C & 0\end{array}\right] x,
\end{equation}
with matrices $A$, $B$ and $C$ of the appropriate dimensions, which are $\left(A,\left[\begin{array}{cc}C &0\end{array}\right]\right)$ detectable in the sense of Definition \ref{def:detectable} coincide with the detectability notion for linear systems \cite[Section 16.3]{hespanha2018linear}. In which case, the observer \eqref{eq:detectable_observer} for linear system \eqref{eq:linear_sys} takes the form
\begin{equation} \label{eq:linear_obs}
    \dot{\hat{x}} = A \hat{x} + L(y-\hat{y}), \qquad
    \hat{y} = \left[\begin{array}{cc}C & 0\end{array}\right] \hat{x},
\end{equation}
and the detectability of the pair $\left(A,\left[\begin{array}{cc}C & 0\end{array}\right]\right)$ implies the existence of a linear function $l(y-\hat{y})=L(y-\hat{y})$ where $L$ is a matrix (also known as the observer gain matrix), such that property \eqref{eq:error_detectable_observer} holds. The observer \eqref{eq:linear_obs} is also known as the Luenberger observer.

Hence, if the closed-loop system \eqref{eq:closed_loop} is detectable, the adversary can estimate the controller's state $x_c$ (Objective \ref{obj:estimate}) by only monitoring the sensor measurements $y$ without manipulating them (i.e., $a(t)=0$ for all $t\geq 0$) and thereby remaining stealthy (Objective \ref{obj:stealth}) under Assumption \ref{assum:closed_loop_AS}. We summarise this in Proposition \ref{prop:detect} below.

\begin{prop} \label{prop:detect}
Consider the closed-loop system \eqref{eq:closed_loop} and adversary model under Assumptions \ref{assum:closed_loop_AS} and \ref{assum:adversary}, respectively. If the closed-loop system \eqref{eq:closed_loop} is detectable, then the adversary achieves Objectives \ref{obj:estimate} and \ref{obj:stealth} using \eqref{eq:detectable_observer} with $a(t)=0$, for all $t\geq 0$, in the sense that
\begin{align} \label{eq:detectable_result}
    |\hat{x}(t)-x(t)| & \leq \beta_{\hat{x}}(|\hat{x}(0)-x(0)|,t), \\
    |x(t)| & \leq \beta_x(|\hat{x}(0)-x(0)|,t), \;  \forall t\geq 0,
\end{align}
for all initial conditions $\hat{x}(0)$, $x(0)\in\mathbb{R}^{n_x+n_c}$ satisfying $|\hat{x}(0)|\leq \Delta_x$ and $|x(0)|\leq \Delta_x$, and $\beta_{\hat{x}} \in \mathcal{KL}$ and $\beta_{x} \in \mathcal{KL}$ comes from Definition \ref{def:detectable} and Assumption \ref{assum:closed_loop_AS}, respectively. \hfill $\Box$ 
\end{prop}
As seen in \eqref{eq:detectable_result}, the adversary performs better than the stated objectives by achieving asymptotic convergence of the estimates $\hat{x}_c$ to the controller's state $x_c$ and the closed-loop system \eqref{eq:closed_loop} preserves the inherent SG-AS property from Assumption \ref{assum:closed_loop_AS}. 

This setup was studied in discrete-time for an LTI closed-loop system in \cite{umsonst2021confidentiality} in the presence of Gaussian process and measurement noise where a time-varying Kalman filter is proposed as the optimal controller's state estimator. Here, we do not consider the presence of noise, but leveraging noise to preserve the confidentiality of the controller will be an important future endeavour of this work. 

The crucial assumption in \cite{umsonst2021confidentiality} is the detectability of the closed-loop system. To the best of our knowledge, no results exist in the literature for when the closed-loop system \eqref{eq:closed_loop} is not detectable according to Definition \ref{def:detectable}. Hence, the novelty of this work is in showing that controller confidentiality can be breached in the absence of closed-loop detectability. The rest of the paper is dedicated to this unexplored aspect.

%% file: result-estimation-not.tex
\section{Closed-loop system \eqref{eq:closed_loop} is \textit{NOT} detectable} \label{sec:result_estimate_not}
When the closed-loop system \eqref{eq:closed_loop} is \textit{NOT} detectable, the stealthy estimation of the controller's state $x_c$ can be achieved through manipulating the sensor measurement $y$ by way of the attack signal $a$, modelled by \eqref{eq:system_output}. The compromised sensor $y$ is used to probe the closed-loop system \eqref{eq:closed_loop} periodically within the time interval $[kT,(k+1)T]$, $k\in\mathbb{N}_{\geq 0}$, for a short period of time $t^*>0$ such that the controller's state $x_c$ can be estimated within some desired margin of error during the probing interval of $[kT,kT+t^*]$ and with bounded error for the remainder of the interval. The probing however, may lead to detection by the operator, and hence is only held sufficiently long, such that the closed-loop system \eqref{eq:closed_loop} remains practically stable, i.e., stealth is maintained according to \eqref{eq:stealth_obj}. 

To this end, we require a modification of an observability notion  first introduced in \cite{shim2003asymptotic} where we need to apply an open-loop probing signal $y^*$ for the closed-loop system \eqref{eq:closed_loop} within a finite time interval such that its states can be estimated. Adopting the same terminology as in \cite{shim2003asymptotic}, such an observability notion is defined as follows:

\begin{defn}[Robust observabilty] \label{def:obs}
 System \eqref{eq:closed_loop}, \eqref{eq:system_output} is semiglobal $q$-robust observable (SG$q$-RO) for  $t\in[0,t^*]$,  $t^*>0$, if, for each $\Delta_x \geq 0$, there exist an integer $q\in\mathbb{N}_{\geq 1}$, a $\mathcal{C}^{q+1}$ function $y^*:[0,t^*]\to\mathbb{R}^{n_y}$ and a function $\Psi:\mathbb{R}^{2(q+1)n_y}\to \mathbb{R}^{n_x}$ such that for the probed system
\begin{equation} \label{eq:probed_system}
    \dot{x}=f(x,y^*)=\left(\begin{array}{c} f_p(x_p,\kappa(x_c,y^*)) \\ f_c(x_c, y^*) \end{array}\right), \; y=h(x_p),
\end{equation}
with initial condition $|x(0)|\leq \Delta_x$, the following is satisfied for $t\in[0,t^*]$,
\begin{itemize}
    \item the solution $x(t)$ to \eqref{eq:probed_system} exists,
    \item the function $\Psi$ maps the measurement $h(x_p)$ and the probing signal $y^*$ as well as their derivatives to the solution $x(t)$ as follows   $$x(t)=\Psi(Y(t),Y^*(t)),$$ where $Y:=(h(x_p), L_{f_p}{h}(x_p), \dots, L_{f_p}h^{(q)}(x_p))$ and $Y^*:=(y^*,\dot{y}^*,\dots,{y^{*}}^{(q)})$, where $L_{f_p}h^{(q)}(x_p)$ denotes the $q$-th time derivative of $h(x_p)$,
    \item  there exists $\rho_{\Psi}\in\mathcal{K}_{\infty}$ such that
    $$ \left| \Psi(\widehat{Y},Y^*) - \Psi(Y,Y^*) \right| \leq \rho_{\Psi}\left(\left|\widehat{Y}-Y\right|\right). $$
\end{itemize} \vspace{-1em}
\hfill $\Box$
\end{defn}
For examples of systems which are SG$q$-RO and on how to construct the probing signal $y^*$, the reader is referred to the origin of this observability notion in \cite{shim2003asymptotic}. A consequence of the SG$q$-RO property of system \eqref{eq:closed_loop}, \eqref{eq:system_output} is that the sensor measurement $h(x_p)$ needs to be sufficiently smooth. For the proposed adversarial scheme to work, we further require the following.
\begin{assum}\label{assum:compact_smooth_output}
 Suppose that the closed-loop system \eqref{eq:closed_loop} is SG$q$-RO. For a given $t^*>0$, there exist compact sets $\mathcal{H}_{q}$ and $\mathcal{H}_{q+1}$ such that for all  $t\in[0,t^*]$,
     \begin{align}
         \left(L_{f_p}{h}(x_p(t)), \dots, L_{f_p}h^{(q)}(x_p(t)))  \right) & \in \mathcal{H}_{q}, \textrm{ and } \nonumber \\
          L_{f_p}h^{(q+1)}(x_p(t))) & \in \mathcal{H}_{q+1}.
     \end{align} 
     \hfill $\Box$
\end{assum}

The proposed adversarial strategy which ensures that the adversary's estimate of the controller state $\hat{x}_c$ converges to a neighborhood of the controller's state $x_c$ and the closed-loop system \eqref{eq:closed_loop} is semiglobal practical stable  (maintaining stealth), takes the following form under the assumption that the closed-loop system \eqref{eq:closed_loop} is SG$q$-RO.

First, a probing duration $t^*>0$ is chosen and then a suitable total duration $T>t^*$ is selected. Then each time interval $[kT,(k+1)T]$, for $k\in\mathbb{N}_{\geq 0}$, is subdivided into a probing interval $\mathcal{T}_{k}:=[kT,kT+t^*)$ and non-probing interval ${\overline{\mathcal{T}}}_{k}:=[kT+t^*,(k+1)T)$. During the probing interval $\mathcal{T}_{k}$, the adversary probes the system for a short duration $t^*>0$ by compromising the sensor measurements $y$ via the attack signal $a$. After which, the adversary stops probing to maintain stealth during $\overline{\mathcal{T}}_k$ such that,
\begin{equation} \label{eq:probe_y}
    y(t) = \left\{\begin{array}{ll} y^*(t-kT), & t\in\mathcal{T}_{k},  \\ h(x_p(t)), & t\in\overline{\mathcal{T}}_{k},  \end{array}\right.
\end{equation}
and the resulting closed-loop system is
\begin{equation} \label{eq:closed_loop_probe}
    \dot{x}(t) = \left\{\begin{array}{ll} f(x(t),y^*(t-kT)), & t\in\mathcal{T}_{k}, \\ f(x(t),h(x_p(t))), & t\in\overline{\mathcal{T}}_{k}.  \end{array}\right.
\end{equation}
By the probing procedure of \eqref{eq:probe_y}, the controller's state estimate $\hat{x}_{c}$ is obtained via
\begin{align} \label{eq:estimate_y}
    \dot{\widehat{Y}}(t) & = \left\{\begin{array}{ll} \widehat{A} \widehat{Y}(t) + \theta \Delta_\theta \widehat{H}\left( h(x_p(t))-\widehat{C}\widehat{Y}(t) \right), & t\in\mathcal{T}_{k}, \\ 0, & t\in\overline{\mathcal{T}}_{k},  \end{array}\right. 
\end{align}
where 
    $\widehat{Y}\in\mathbb{R}^{(q+1)n_y}$,
    $\widehat{C}=\left[\begin{array}{cc} \mathbb{I}_{n_y} & 0_{n_y \times (q+1)n_y} \end{array} \right]$,
    $\widehat{H}=\left[\begin{array}{cccc} a_1 \mathbb{I}_{n_y} & a_2 \mathbb{I}_{n_y} & \dots & a_{q+1} \mathbb{I}_{n_y} \end{array} \right]^{T}$ with $a_i$ chosen such that the polynomial $s^{q+1}+a_1 s^{q} + a_2 s^{q-1} + \dots + a_{q+1}$ is Hurwitz,
    $\Delta_\theta = \textrm{diag}\left(\mathbb{I}_{n_y},\theta\mathbb{I}_{n_y}, \theta^2 \mathbb{I}_{n_y},\dots, \theta^q \mathbb{I}_{n_y}\right)$ with a constant tuning parameter $\theta\geq 1$, and
    $\widehat{A}=\left[\begin{array}{cc} 0_{qn_y \times n_y} & \mathbb{I}_{qn_y} \\ 0_{n_y\times n_y} & 0_{n_y\times qn_y} \end{array} \right].$
\vspace{1em}

The initialisation of \eqref{eq:estimate_y} is chosen to be
\begin{equation} \label{eq:estimate_initial}
    \widehat{Y}(kT) \in \mathcal{Y}(h(x_p(kT))),
\end{equation}
where $\mathcal{Y}(r):=\left\{\widehat{Y}\in\mathbb{R}^{(q+1)n_y}: \left| \widehat{Y} - r \right| \leq  \epsilon_y, \epsilon_y>0  \right\}$.
\vspace{1em}

The adversary can then obtain an estimate of the controller's state as follows
\begin{align} \label{eq:estimate_xc}
    \hat{x}(t) = \left(\begin{array}{c}\hat{x}_p(t) \\ \hat{x}_{c}(t) \end{array}\right)  = \Psi\left(\widehat{Y}(t),Y^*(t-kT)\right),  t\in \mathcal{T}_{k} \cup \bar{\mathcal{T}}_{k},
\end{align}
where $q$, and $\Psi$ come from the assumption that the closed-loop system is SG$q$-RO as defined in Definition \ref{def:obs}. The probing duration $t^*>0$ is dictated by the robustness of the closed-loop system \eqref{eq:closed_loop} such that it is semiglobally practically stable (Objecive 2: maintaining stealth).

\section{Main result} \label{sec:main_result}
The proposed probing scheme was inspired by the dual mode sampled-data output feedback control strategy in \cite{shim2003asymptotic}. Consequently, elements of the proof of Theorem \ref{thm:stealthy_estimation} follow that of \cite{shim2003asymptotic} and \cite{kellett2004further} with some modifications as our resulting closed-loop system \eqref{eq:closed_loop}, \eqref{eq:system_output} does not have a sample-and-hold input. The adversary's estimate of the controller's state does employ a sample-and-hold observer \eqref{eq:estimate_y}, \eqref{eq:estimate_initial}, \eqref{eq:estimate_xc}, but is not employed in the closed-loop system \eqref{eq:closed_loop}.  

In the sequel, we will pave the way towards our main result (Theorem \ref{thm:stealthy_estimation}) in Section \ref{sub:thm} by addressing the fulfillment of Objective 1 and 2 in Sections \ref{sub:estimate} and \ref{sub:stealth}, respectively. 

\subsection{Objective 1: controller's state estimation} \label{sub:estimate}
\begin{prop} \label{prop:estimate}
Given $t^*>0$, consider the closed-loop system \eqref{eq:closed_loop_probe}, \eqref{eq:probe_y} that is SG$q$-RO and satisfies Assumption \ref{assum:closed_loop_AS} and \ref{assum:compact_smooth_output}, the adversary model under Assumption \ref{assum:adversary}, the estimator \eqref{eq:estimate_y}, \eqref{eq:estimate_initial} and the controller's state estimate \eqref{eq:estimate_xc}. For all $K_{\tilde{x}}>0$, there exist $\theta \geq 1$ and $\sigma_{\tilde{x}}\in\mathcal{K}_{\infty}$ such that for all $k\in\mathbb{N}_{\geq 0}$,
\begin{align} \label{eq:xhat_T}
    &|\hat{x}(kT+t^*)-x(kT+t^*)| \leq K_{\tilde{x}}, & \nonumber \\
    & |\hat{x}(t)-x(t)| \leq \sigma_{\tilde{x}}(\Delta_{e,k}\theta^{q-1}),  & \forall t\in \mathcal{T}_{k},
\end{align}
where $\Delta_{e,k}:={\max}\{|Y_a-Y_b|, \, Y_a,Y_b \in \mathcal{Y}(h(x_p(kT)) \}$. Further, suppose $|x(kT+t^*)|\leq {\Delta_x}$. Then, there exists $\overline{\sigma}_{\tilde{x}}\in\mathcal{K}_{\infty}$ such that
\begin{align} \label{eq:xhat_Tbar}
    |\hat{x}(t)-x(t)| &\leq K_{\tilde{x}} +  \Delta_{x} + \overline{\sigma}_{\tilde{x}}(\Delta_x),    & \forall t\in \overline{\mathcal{T}}_{k},
\end{align}
and for all $\epsilon_{\tilde{x}}>0$, there exists $T>0$ such that
\begin{align} \label{eq:xhat_Tbar_end}
    |\hat{x}((k+1)T)-x((k+1)T)| &\leq K_{\tilde{x}} + \Delta_{x} + \epsilon_{\tilde{x}}.
\end{align}
\hfill $\Box$
\end{prop}
The proof of Proposition \ref{prop:estimate} employs the following two lemmas, each addressing the conditions needed and the convergence guarantees obtained for the probing interval $\mathcal{T}_{k}$ and the non-probing interval $\overline{\mathcal{T}}_{k}$ within each time interval $[kT, (k+1)T]$, for $k\in\mathbb{N}_{\geq 0}$. For clarity, the lemmas will be written for the case where $k=0$, and the results carry over straightforwardly to $k\in \mathbb{N}_{\geq 1}$ which we state in a remark that follows each lemma. The proofs of the lemmas can be found in the Appendix.

\begin{lem}[the probing interval $\mathcal{T}_{0}$] \label{lem:probe}
Under the same hypothesis as Proposition \ref{prop:estimate} for $t\in\mathcal{T}_{0}$, for all $K_{\tilde{x}}>0$, there exist  $\theta \geq 1$ and $\sigma_{\tilde{x}}\in\mathcal{K}_{\infty}$ such that
\begin{align} \label{eq:xhat_T_0}
    |\hat{x}(t^*)-x(t^*)| &\leq K_{\tilde{x}}, \nonumber \\
    |\hat{x}(t)-x(t)| &\leq \sigma_{\tilde{x}}(\Delta_{e,0}\theta^{q-1}),  \qquad \forall t\in \mathcal{T}_{0}, 
\end{align}
where $\Delta_{e,0}:=\max \left\{ \left|\widehat{Y}_{a} - \widehat{Y}_{b}\right| : \widehat{Y}_{a}, \widehat{Y}_{b} \in \mathcal{Y}(h(x_p(0))) \right\}$. \hfill $\Box$
\end{lem}

\begin{rem}
    Given $T>0$, Lemma \ref{lem:probe} also holds true for subsequent probing intervals $\mathcal{T}_{k}$, $k\in\mathbb{N}_{\geq 0}$, where \eqref{eq:xhat_T_0} becomes
    \begin{align} \label{eq:xhat_T_k}
    |\hat{x}(kT+t^*)-x(kT+t^*)| &\leq K_{\tilde{x}}, \nonumber \\
    |\hat{x}(t)-x(t)| &\leq \sigma_{\tilde{x}}(\Delta_{e}\theta^{q-1}),  \; \forall t\in \mathcal{T}_{k}, 
\end{align}
where $\Delta_{e,k}:={\max} \left\{ \left|\widehat{Y}_{a} - \widehat{Y}_{b}\right| : \widehat{Y}_{a}, \widehat{Y}_{b} \in \mathcal{Y}(h(x_p(kT))) \right\}$. \hfill $\Box$
\end{rem}

\begin{lem}[the non-probing interval $\overline{\mathcal{T}}_{0}$] \label{lem:non_probe}
Under the same hypothesis as Proposition \ref{prop:estimate} for $t\in\overline{\mathcal{T}}_{0}$, suppose $|x(t^*)|\leq {\Delta_x}$. Then, there exists $\bar{\sigma}_{\tilde{x}}\in\mathcal{K}_{\infty}$ such that 
\begin{equation} \label{eq:xhat_T_0_non}
    \left|\hat{x}(t) - x(t) \right| \leq K_{\tilde{x}} +  \Delta_{x} + \overline{\sigma}_{\tilde{x}}(\Delta_x), \qquad \forall t\in \overline{\mathcal{T}}_{0},
\end{equation}
and for all $\epsilon_{\tilde{x}}>0$, there exists $T>0$ such that
\begin{equation} \label{eq:xhat_T_0_T_non}
    \left|\hat{x}(T) - x(T) \right| \leq K_{\tilde{x}} + \Delta_{x} + \epsilon_{\tilde{x}}.
\end{equation} \hfill $\Box$
\end{lem}
\begin{rem}
For subsequent non-probing intervals $\overline{T}_{k}$, where $k\in\mathbb{N}_{\geq 0}$, the estimation error bounds in \eqref{eq:xhat_T_0_non} and \eqref{eq:xhat_T_0_T_non} become
\begin{equation} \label{eq:xhat_T_k_non}
    \left|\hat{x}(t) - x(t) \right| \leq K_{\tilde{x}} +  \Delta_{x} + \overline{\sigma}_{\tilde{x}}(\Delta_x), \qquad \forall t\in \overline{\mathcal{T}}_{k},
\end{equation}
\begin{equation} \label{eq:xhat_T_k_T_non}
    \left|\hat{x}(kT) - x(kT) \right| \leq K_{\tilde{x}} + \Delta_{x} + \epsilon_{\tilde{x}}.
\end{equation} \hfill $\Box$
\end{rem}

The periodic interval $T>0$ is chosen a-priori based on the desirable margin of estimation error which can be made small up to $K_{\tilde{x}}+\Delta_x$, which is the sum of the error margin at the end of the probing period and the size of the initial condition of the non-probed closed loop system \eqref{eq:closed_loop}, \eqref{eq:system_output}.

\subsection{Objective 2: maintaining stealth} \label{sub:stealth}
The analysis that allows the adversary to maintain stealth hinges on the fact that the uncompromised closed-loop system \eqref{eq:closed_loop} is inherently semiglobal asymptotically stable (Assumption \ref{assum:closed_loop_AS}), which by application of a converse Lyapunov theorem (see \cite[Theorem 4.14]{khalil2002nonlinear}, for instance), there exists a $C^1$ closed-loop control Lyapunov function $V:\mathbb{R}^{n_p+n_c}\to\mathbb{R}_{\geq 0}$ such that there exist $\alpha_1$, $\alpha_2$  $\alpha_3\in\mathcal{K}_{\infty}$ where
\begin{enumerate}[(C1)]
    \item $\alpha_{1}\left(|x|\right) \leq V(x) \leq \alpha_{2}\left(|x|\right)$,
    \item $\langle\nabla V(x),f(x,h(x_p))\rangle \leq -\alpha_3(V(x))$, for $x\in\mathcal{V}(R)$, with $R:=\alpha_1(\Delta_x)$,
\end{enumerate}
where we define $\mathcal{V}(r,R):=\{x\in\mathbb{R}^{n_p+n_c}:r\leq V(x) \leq R\}$ and denote $\mathcal{V}(-\infty,R)$ by $\mathcal{V}(R)$.

By straightforward application of \cite[Lemma 4.4]{khalil2002nonlinear} and the comparison lemma \cite[Lemma 3.4]{khalil2002nonlinear}, a consequence of (C2) is stated below, 
\begin{enumerate}[(C2')]
    \item there exists $\beta_{V}\in\mathcal{KL}$ such that $V(x(t))\leq \beta_{V}(V(x(0)),t)$, for all $t\geq 0$, for $x\in\mathcal{V}(R)$, with $R:=\alpha_2(\Delta_x)$.
\end{enumerate}

We will use this closed-loop control Lyapunov function $V$ to show that the probed closed-loop system \eqref{eq:closed_loop_probe} is semiglobally practically stable. To do so, the given control Lyapunov function $V$ and the vector field $f$ of the probed closed-loop system \eqref{eq:closed_loop_probe} have to possess the following properties.
\begin{assum} \label{assum:Vf}\ 
\begin{enumerate}[(V1)]
    \item There exists $\rho\in\mathcal{K}_{\infty}$ such that $|V(x)-V(w)| \leq \rho(|x-w|)$ for all $x$, $w\in\mathcal{V}(R,R+R_m)$, for some $R_m>0$.
    \item There exists a constant $F^*>0$ such that $|f(x,y^*)|\leq F^*$, for all $x\in\mathcal{V}(R+R_m)$.
    \item There exists a constant $F>0$ such that $|f(x,h(x_p))|\leq F$, for all $x\in\mathcal{V}(R+R_m)$.
\end{enumerate}
\end{assum}

We employ a key lemma in showing an $\mathcal{L}_{1}$-type robustness with respect to additive disturbance for the first interval $[0,T]$. The results carry over to subsequent intervals $[kT,(k+1)T]$ which will be stated in a remark below. The proof of the lemma below is inspired by \cite{kellett2004further} and can be found in the Appendix.
\begin{lem}\label{lem:L1_robust}
    Consider the closed-loop system \eqref{eq:closed_loop} under Assumption \ref{assum:closed_loop_AS} and (V1) of Assumption \ref{assum:Vf}. Given $T>0$ and $0<r<R$, consider
    \begin{equation}\label{eq:perturb}
        \dot{x}(t)=f(x(t),h(x_p(t)))+d(t), \; \forall t\in[0,T], 
    \end{equation}
    and for all $x(0)\in\mathcal{V}(R)$. Let $\sigma\in[0,R-r)$. If $d(t)$ satisfies
    \begin{equation} \label{eq:perturb_cond}
        \underset{t\in[0,T)}{\max} \left|\int_{0}^{T} d(s) ds \right| \leq \rho^{-1}(\sigma) e^{-\bar{L}T},
    \end{equation}
    where $\bar{L}:=l_x+l_yl_h >0$, where $l_x>0$ and $l_y>0$ are the Lipschitz constants of the function $f$ with respect to each of its arguments\footnote[3]{Since the functions $f_p$, $f_c$ and $\kappa$ are all locally Lipschitz in their arguments, the function $f$ is also locally Lipschitz in its arguments \label{foot:lip}}, respectively, and $l_h>0$ is the Lipschitz constant of the function $h$.
    Then the solution $x(t)$ to \eqref{eq:perturb} exists and satisfies
    \begin{equation} \label{eq:perturb_results}
        V(x(t)) \leq \beta_{V}(V(x(0),t)+\sigma, \; \forall t\in[0,T], 
    \end{equation}
    and for all $x(0)\in\mathcal{V}(R)$. \hfill $\Box$
\end{lem}

\begin{rem}
The result of Lemma \ref{lem:L1_robust} is applicable to subsequent time intervals $[kT,(k+1)T]$, where \eqref{eq:perturb_results} is replaced with
\begin{equation} \label{eq:perturb_results}
        V(x(t)) \leq \beta_{V}(V(x(kT),t-kT)+\sigma, 
    \end{equation}
    for all $t\in[kT,(k+1)T]$ and $x(kT)\in\mathcal{V}(R)$. \hfill $\Box$
\end{rem}

Rewriting our probed closed-loop system \eqref{eq:closed_loop_probe}
in perturbed form, we get the perturbed system \eqref{eq:perturb} for $t\in \mathcal{T}_{k}\cup \bar{\mathcal{T}}_{k}=[kT,(k+1)T]$, with
\begin{equation}
    d(t) = \left\{ \begin{array}{cc}
         f(x(t),y^*(t)) - f(x(t),h(x_p(t))), & t\in\mathcal{T}_{k}, \\
         0, & t\in\bar{\mathcal{T}}_{k}. 
    \end{array} \right.
\end{equation}
To apply Lemma \ref{lem:L1_robust}, we observe that
{\small
\begin{align}
    \underset{t\in[kT,(k+1)T)}{\max}& \left|\int_{kT}^{(k+1)T} d(s) ds \right| \nonumber \\
    & \leq \left| \int_{kT}^{kT+t^*} f(x(s),y^*(s))-f(x(s),h(x_p(s)) ds \right| \nonumber \\
    & \leq  \int_{kT}^{kT+t^*} |f(x,y^*)|+|f(x,h(x_p))| ds \nonumber \\
    & \leq (F^*+F)t^*,
\end{align}
}
where we got the last inequality using (V2) and (V3).

Therefore, with \eqref{eq:stealth_cond} in Proposition \ref{prop:stealth} below, the hypothesis of Lemma \ref{lem:L1_robust} is fulfilled. We can then apply Lemma \ref{lem:L1_robust} to prove the following proposition. 

\begin{prop} \label{prop:stealth}
Consider the probed closed-loop system \eqref{eq:closed_loop_probe} under Assumption \ref{assum:closed_loop_AS}, \ref{assum:adversary}, \ref{assum:Vf}. Given $T>0$, suppose there exist $t^*<T$ and $r\in(0,R)$ satisfying
\begin{equation} \label{eq:stealth_cond}
    r \leq \beta_{V}(R,t^*), \qquad
    (F+F^*)t^* \leq  \rho^{-1}(\sigma) e^{-\bar{L}T},
\end{equation}
with $\sigma \in [0,R-r)$ and $\bar{L}>0$ is as defined in Lemma  \ref{lem:L1_robust}. Then, the solution $x(t)$ to the probed closed-loop system \eqref{eq:closed_loop_probe} exists and satisfies the following for all $t\in[kT,(k+1)T)$, $k\in\mathbb{N}_{\geq 0}$, 
\begin{align} \label{eq:stealth_res}
    & V(x(t))  \leq \beta_{V}(V(x(kT)),t-kT)+\sigma,   \nonumber \\
    & V(x(kT+t^*))  \leq R, \qquad \forall x(kT)\in\mathcal{V}(R).
\end{align} \hfill $\Box$
\end{prop}

\subsection{Achieving Objective 1 and 2: stealthy estimation} \label{sub:thm}
We are now ready to state our main result.
\begin{thm} \label{thm:stealthy_estimation}
    Suppose the closed-loop system \eqref{eq:closed_loop} and \eqref{eq:system_output} is SG$q$-RO (Defnition \ref{def:obs}) and satisfies Assumptions \ref{assum:closed_loop_AS}, \ref{assum:compact_smooth_output}, and the adversarial model satisfies Assumption \ref{assum:adversary}. Then, the adversary can employ the probing scheme of  \eqref{eq:probe_y}, \eqref{eq:closed_loop_probe}, \eqref{eq:estimate_xc}, \eqref{eq:estimate_y} initialised according to \eqref{eq:estimate_initial} to achieve stealthy estimation, if 
    \begin{enumerate}[(i)]
        \item Assumption \ref{assum:Vf} holds, and
        \item For any $\Delta_{x}>0$, $K_{\tilde{x}}>0$ and $\epsilon_{\tilde{x}}>0$, there exist $t^*>0$, $T>t^*$ and $r\in(0,R)$ satisfying \eqref{eq:stealth_cond} and \eqref{eq:choose_T} with $R:=\alpha_1(\Delta_x)$, $\sigma \in [0,R-r)$, $\beta_{V}\in\mathcal{KL}$ from (C2'), $\rho\in\mathcal{K}_{\infty}$, $F>0$, $F^*>0$ from Assumption \ref{assum:Vf}.
    \end{enumerate} 
    Stealthy estimation is achieved in the sense that the objectives are fulfilled in the following manner: 
    \begin{itemize}
        \item \emph{Objective 1 (estimation of the controller's state)}: there exist $\theta\geq 1$ (chosen according to \eqref{eq:choose_theta_1} and \eqref{eq:choose_theta_2}), $\sigma_{\tilde{x}}$, $\overline{\sigma}_{\tilde{x}}\in\mathcal{K}_{\infty}$ such that for all $t\in[kT,(k+1)T)$, $k\in\mathbb{N}_{\geq 0}$,
        \begin{align} \label{eq:estimate_main}
            & |\hat{x}(t)-x(t)| \nonumber \\
            &\qquad \leq \max\{\sigma_{\tilde{x}}(\Delta_{e,k}\theta^{q-1}),  K_{\tilde{x}}+\Delta_x+\bar{\sigma}_{\tilde{x}}(\Delta_x) \}, \nonumber \\
            & |\hat{x}(kT+t^*)-x(kT+t^*)| \leq K_{\tilde{x}}, \nonumber \\
            & |\hat{x}(kT)-x(kT)| \leq K_{\tilde{x}} + \Delta_{x}+\epsilon_{\tilde{x}}, 
        \end{align}
        with $\Delta_{e,k}:={\max} \left\{ \left|\widehat{Y}_{a} - \widehat{Y}_{b}\right| : \widehat{Y}_{a}, \widehat{Y}_{b} \in \mathcal{Y}(h(x_p(kT))) \right\}$.
        \item \textit{Objective 2 (maintaining stealth)}: there exists $K_{x}=K_{x}(\Delta_x)>0$ such that for all $t\in[kT,(k+1)T)$, $k\in\mathbb{N}_{\geq 0}$,
        \begin{equation} \label{eq:stealth_main}
            |x(t)| \leq K_x.
        \end{equation}
    \end{itemize} \hfill $\Box$
\end{thm}
\begin{proof}
Let $k\in\mathbb{N}_{\geq 0}$ and $t\in[kT,(k+1)T)$. We will employ Proposition \ref{prop:estimate} and \ref{prop:stealth}. Since \eqref{eq:stealth_cond} and \eqref{eq:choose_T} hold, we first apply Proposition \ref{prop:stealth} to obtain \eqref{eq:stealth_res}. From the first inequality of \eqref{eq:stealth_res} and using (C1), we obtain 
\begin{align}
    \alpha_1(|x(t)|) \leq \beta_{V}(R,0)+R-r \nonumber \\
    \implies |x(t)| \leq \alpha_1^{-1}(\beta_V(R,0)+R).
\end{align}
Since $R:=\alpha_1(\Delta_x)$, we achieve \eqref{eq:stealth_main} with $K_x:=\alpha_1^{-1}(\beta_V(\alpha_1(\Delta_x),0)+\alpha_1(\Delta_x))$.

Next, we see that condition requiring $|x(kT+t^*)|\leq \Delta_x$ in Proposition \ref{prop:estimate} is fulfilled with the second inequality of \eqref{eq:stealth_res} using (C1). By applying Proposition \ref{prop:estimate}, we obtain \eqref{eq:estimate_main} as desired. 
\end{proof}

%% file: conclusion.tex
\section{Conclusion and future work} \label{sec:conclude}
We have shown that the confidentiality of the controller's states can easily be breached stealthily when the closed-loop system is detectable. In this scenario, the adversary merely needs to gather measurement data from the sensors. In the absence of a detectable closed-loop system, but under a relaxed robust observability property, the adversary may employ a time-shared probing scheme by manipulating the sensor data to estimate the controller's state within a desired margin of error during the probing period and with bounded error during the non-probing period. Additionally, stealth is maintained in the sense that the closed-loop system remain semiglobally practically stable. Future work includes devising defence mechanisms to obfuscate the adversary's estimate of the controller's state.

%% file: appendix.tex
\subsection{Proof of Lemma \ref{lem:probe}} \label{sec:proof_lem_probe}
Let $t\in\mathcal{T}_0$. A key observation is that systems which are SG$q$-RO have the following dynamics in the $Y$-coordinate.
\begin{equation} \label{eq:Ydot}
    \dot{Y} = \widehat{A}Y + \widehat{B}L_{f_p}h^{(q+1)}(x_p(t))),
\end{equation}
where $\widehat{A}$ is as defined for the estimator \eqref{eq:estimate_y} and $\widehat{B}=( 0_{qn_y \times n_y}^T, \mathbb{I}_{n_y}^T )^T$. In fact, this observation induces the estimator design in  \eqref{eq:estimate_y}. From \eqref{eq:Ydot} and \eqref{eq:estimate_y} The dynamics of the estimator mismatch $e:=\hat{Y}-Y$ is
\begin{equation}
    \dot{e} = (\widehat{A} + \theta \Delta_{\theta} \widehat{H}\widehat{C})e - \widehat{B}L_{f_p}h^{(q+1)}(x_p(t))).
\end{equation}

We rescale the estimator mistmatch as $z=\Delta_{\theta}^{-1} e$, which has the dynamics
\begin{eqnarray} \label{eq:scaled_error_z}
 \dot{z} & = & (\Delta_{\theta}^{-1} \widehat{A} \Delta_{\theta}  - \theta \widehat{H} \widehat{C} \Delta_{\theta} )z - \Delta_{\theta}^{-1} L_{f_p}h^{(q+1)}(x_p(t))) \nonumber \\
 & = & \theta \left( \widehat{A}-\widehat{H}\widehat{C} \right) z - \Delta_{\theta}^{-1} L_{f_p}h^{(q+1)}(x_p(t))),
\end{eqnarray}
where due to the structure of the $\widehat{A}$, $\widehat{C}$ and $\Delta_{\theta}$, we have used $\Delta_{\theta}^{-1} \widehat{A} \Delta_{\theta}=\theta \widehat{A}$ and $\widehat{C}\Delta_{\theta}=\widehat{C}$ to obtain the resulting dynamics of the scaled estimation mismatch system $z$.

Since the matrix $\widehat{A}-\widehat{H}\widehat{C}$ is Hurwitz, there exist a matrix $P=P^T>0$ and scalar $\nu>0$ satisfying
\begin{equation} \label{eq:PA}
    P \left( \widehat{A}-\widehat{H}\widehat{C} \right) + \left( \widehat{A}-\widehat{H}\widehat{C} \right)^{T} P \leq -\nu \mathbb{I}_{(q+1)n_y}.
\end{equation}

Using a candidate Lyapunov function $W(z):=z^{T}Pz$ with $P$ satisfying \eqref{eq:PA}, its time derivative along the solutions to \eqref{eq:scaled_error_z} is
\begin{align}
  \dot{W}(z) = &   \theta z^T \left( P \left( \widehat{A}-\widehat{H}\widehat{C} \right) + \left( \widehat{A}-\widehat{H}\widehat{C} \right)^{T} P \right) z \nonumber \\
  &+ 2 z^T P \Delta_{\theta}^{-1} L_{f_p}h^{(q+1)}(x_p(t))).
\end{align}
We note that the last term satisfies $z^T P \Delta_{\theta}^{-1} L_{f_p}h^{(q+1)}(x_p)) \leq |z||P \Delta_{\theta}^{-1} L_{f_p}h^{(q+1)}(x_p))|$. Under Assumption \ref{assum:compact_smooth_output}, we obtain $|L_{f_p}h^{(q+1)}(x_p))|\leq \bar{\phi}$, where $\bar{\phi}>0$. Hence, $|P \Delta_{\theta}^{-1} L_{f_p}h^{(q+1)}(x_p))| \leq \frac{\bar{\lambda}\bar{\phi}}{\theta^{q}}$, where $\bar{\lambda}:=\lambda_{\max}(P)$ Therefore, in conjuction with \eqref{eq:PA}, the time derivative of $W(z)$ is
\begin{align}
  \dot{W}(z) \leq  &   -\theta \nu |z|^2 +  \frac{2\bar{\lambda}\bar{\phi}}{\theta^{q}} |z| .
\end{align}

Hence, if $\frac{2\bar{\lambda}\bar{\phi}}{\theta^{q}} \leq \frac{\theta \nu}{2} |z|$, then
\begin{align} \label{eq:lem_w_1}
  \dot{W}(z) \leq  &   -\frac{\theta \nu}{2} |z|^2 \leq -\frac{\theta \nu}{2 \underline{\lambda}} W(z),  
\end{align}
where we have used the following fact to obtain the ultimate bound: 
\begin{equation} \label{eq:W_low_high}
 \underline{\lambda}|z|^2 \leq W(z) \leq \bar{\lambda}|z|^2,
\end{equation}
where $\underline{\lambda}:=\lambda_{\min}(P)$.

By the comparison principle, we obtain from \eqref{eq:lem_w_1} that
\begin{equation}
    W(z(t)) \leq \exp\left( - \frac{\theta \nu}{2 \underline{\lambda}} t \right) W(z(0)).
\end{equation}
Applying \eqref{eq:W_low_high}, we conclude that
\begin{equation} \label{eq:z_solution}
    |z(t)| \leq c \exp\left( - \frac{\theta \nu}{4 \underline{\lambda}} t \right) |z(0)|,
\end{equation}
when $|z|\geq \frac{4\bar{\lambda}\bar{\phi}}{\nu \theta^{q+1}}$, where $c:=\sqrt{\bar{\lambda}/\underline{\lambda}}$.

From \eqref{eq:z_solution}, we will choose the estimator parameter $\theta\geq 1$ such that
\begin{enumerate}[(E1)]
    \item when $z(0)\not\in \left\{z: |z|\leq \frac{4\bar{\lambda}\bar{\phi}}{\nu \theta^{q+1}}\right\}=:\Omega$, $z(t)$ converges to $\Omega$ within  finite time and then remains there.
    \item when $z(0)\in\Omega$, $z(t)$ remains in $\tilde{\Omega}:=\left\{z\in\mathbb{R}^{(q+1)n_y}:|z|\leq \rho_{\Psi}^{-1}\left(K_{\tilde{x}} \right) \right\}$, where $\rho_{\Psi}\in\mathcal{K}_{\infty}$ comes from the SG$q$-RO property and $K_{\tilde{x}}>0$ is given.
\end{enumerate}

We first address (E1). In this case, we choose $\theta$ as follows to ensure that $z(t)$ converges to $\Omega$ by $t=t^*$, where $t^*>0$ is assumed to be given. In this case, we choose $\theta \geq 1$ that satisfies
\begin{equation} \label{eq:choose_theta_1}
    \exp\left( - \frac{\theta \nu}{4 \underline{\lambda}} t^* \right) {\Delta_{e,0}} \leq \frac{4 \bar{\lambda}\bar{\phi}}{\theta^{q-1}\nu}. 
\end{equation}
where $\Delta_{e,0}>0$ as defined in Lemma \ref{lem:probe} exists since $\widehat{Y}(0)$ is initialised according to \eqref{eq:estimate_initial}. Notice that by choosing $\theta\geq 1$ to satisfy \eqref{eq:choose_theta_1}, we see from \eqref{eq:z_solution} that although $z(0)\not\in\Omega$, $|z(0)|\leq \Delta_{e,0}/\theta$. Therefore,
\begin{equation} \label{eq:zsol_1}
    |z(t^*)| \leq c \exp\left( - \frac{\theta \nu}{4 \underline{\lambda}} t^* \right) \Delta_{e,0} \leq  c \frac{4 \bar{\lambda}\bar{\phi}}{\theta^{q}\nu}.
\end{equation}

Next, we address (E2). In this case, we choose $\theta \geq 1$ to satisfy
\begin{equation} \label{eq:choose_theta_2}
    \frac{c4\bar{\lambda}\bar{\phi}}{\theta \nu} \leq \rho_{\Psi}^{-1}\left( K_{\tilde{x}}\right).
\end{equation}
With this choice, we see that when $z(0)\in\Omega$, from \eqref{eq:z_solution}, 
\begin{equation}\label{eq:zsol_2}
    |z(t)| \leq c |z(0)| \leq \frac{c4\bar{\lambda}\bar{\phi}}{\nu \theta^{q+1}} \leq   \frac{\rho_{\Psi}^{-1}\left(K_{\tilde{x}}\right)}{\theta^{q}},
\end{equation}
where we obtain the last bound due to our choice of $\theta\geq 1$ according to \eqref{eq:choose_theta_2}. 

Therefore, by choosing $\theta \geq 1$ according to \eqref{eq:choose_theta_1} and \eqref{eq:choose_theta_2}, we ensure that $|z(t^*)|\leq \frac{\rho_{\Psi}^{-1}\left(K_{\tilde{x}}\right)}{\theta^{q}}$ and $|z(t)| \leq c |z(0)| \leq c \Delta_{e,0}/\theta$. 

Next, since $e=\Delta_\theta z$, we get $|e|\leq \theta^{q} |z|$ and
obtain 
\begin{equation}
    |e(t^*)| \leq \rho_{\psi}^{-1}\left(K_{\tilde{x}}\right), \; |e(t)| \leq c \Delta_{e,0} \theta^{q-1}.   
\end{equation}

Finally, by the SG$q$-RO property of \eqref{eq:closed_loop}, \eqref{eq:system_output}, the state estimation error satisfies $|\hat{x}(t)-x(t)|\leq \rho_{\Psi}(|e(t)|)$ for $t\in \mathcal{T}_{0}$. Hence, we obtain \eqref{eq:xhat_T_0} as desired with $\sigma_{\tilde{x}}(r):=cr$. \hfill $\Box$

\subsection{Proof of Lemma \ref{lem:non_probe}}
Given $\Delta_x$, $\epsilon_{\tilde{x}}$, $t^* >0$, choose $T>0$ such that 
\begin{equation} \label{eq:choose_T}
    \beta_{x}(\Delta_x,T-t^*) \leq \epsilon_{\tilde{x}},
\end{equation}
where $\beta_{x}\in\mathcal{KL}$ comes from Assumption \ref{assum:closed_loop_AS}, that the closed loop system \eqref{eq:closed_loop} with \eqref{eq:system_output} semiglobally asymptotically stable for $t\geq t^*$.

Let $t\in\overline{\mathcal{T}}_{0}$. The closed-loop system and state estimate are
\begin{align} \label{eq:system_non_probe}
    \dot{x} & = f(x,h(x_p)), \; y = h(x_p), 
\end{align} \vspace{-5mm}
\begin{equation} \label{eq:state_estimate_non_probe}
    \hat{x}(t)  = \Psi\left(\widehat{Y}(t^*),Y^*(t^*) \right).
\end{equation}
The state estimation error satisfies
\begin{align} \label{eq:estimation_error_bound}
    |\hat{x}(t)-x(t)| & \leq |\hat{x}(t)-x(t^*)+x(t^*)-x(t)| \nonumber \\
    & \leq |\hat{x}(t)-x(t^*)|+|x(t^*)-x(t)|.
\end{align}

A bound on the first term is obtained by the SG$q$-RO property of the closed-loop system \eqref{eq:system_non_probe} at $t=t^*$, when the probing signal $y^*(t^*)$ and observability map $\Psi$ exist. Hence, by Lemma \ref{lem:non_probe}, 
\begin{equation} \label{eq:error_bound_1}
    |\hat{x}(t)-x(t^*)| \leq K_{\tilde{x}}. 
\end{equation}

The second term of \eqref{eq:estimation_error_bound} is bounded as follows
\begin{equation} \label{eq:error_bound_2}
    |x(t^*)-x(t)|    \leq |x(t^*)|+|x(t)|  \leq \Delta_{x} + \beta_{x}(\Delta_{x},t-t^*),
\end{equation}
where we have used the assumption that $|x(t^*)|\leq \Delta_{x}$ and $\beta_x\in\mathcal{KL}$ comes from Assumption \ref{assum:closed_loop_AS} to obtain the final bound. 

Therefore, using \eqref{eq:error_bound_1} and \eqref{eq:error_bound_2} in \eqref{eq:estimation_error_bound}, we obtain
\begin{equation}
    |\hat{x}(t)-x(t)| \leq K_{\tilde{x}} + \Delta_x + \beta_{x}(\Delta_{x},t-t^*).
\end{equation}
We then obtain \eqref{eq:xhat_T_0_non} with  $\overline{\sigma}_{\tilde{x}}(r):=  \beta_{x}(r,0)$ and \eqref{eq:xhat_T_0_T_non} with \eqref{eq:choose_T}, respectively. \hfill $\Box$ 

\subsection{Proof of Lemma \ref{lem:L1_robust}}
Let $t\in[0,T]$ and consider a nominal system 
\begin{equation} \label{eq:nominal}
    \dot{w}=f(w,h(w_p)), \; w(0)\in\mathcal{V}(R), 
\end{equation}
and the perturbed system \eqref{eq:perturb} with $w(0)=x(0)$. The solutions to \eqref{eq:nominal} and \eqref{eq:perturb} are
\begin{align}
    x(t) & = x(0) + \int_{0}^{t} f(x(s),h(x_p(s)))ds + \int_{0}^{t} d(s) ds \nonumber \\
    w(t) & = w(0) + \int_{0}^{t} f(w(s),h(w_p(s)))ds.
\end{align}
Therefore, 
\begin{align}
    |x(t)&-w(t)|\nonumber \\
    &\leq  \int_{0}^{t} |f(x(s),h(x_p(s)))-f(w(s),h(w_p(s)))|ds \nonumber \\
    &\qquad + \left|\int_{0}^{t} d(s) ds \right|, 
\end{align}
and by the Lipshitz property of $f$ (see footnote \ref{foot:lip}) and $h$, we have $l_x$, $l_y>0$ such that 
\begin{align}
    |f(x,h(x_p))&-f(w,h(w_p))|\nonumber \\
    &\leq l_x|x-w|+l_y|h(x_p)-h(w_p)| \nonumber \\
    & \leq l_x|x-w|+l_y l_h |x_p-w_p| \nonumber \\
    & \leq \bar{L} |x-w|,
\end{align}
where we obtain the last inequality due to $|x_p-w_p|\leq |x-w|$ and $\bar{L}$ is as defined in Proposition \ref{prop:stealth}.

Let $\bar{d}:=\underset{t\in[0,T)}{\max} \left|\int_{0}^{T} d(s) ds \right|$. Then,
\begin{equation}
    |x(t)-w(t)|  \leq \int_{0}^{t} \bar{L} |x(s)-w(s)| ds + \bar{d}.  
\end{equation}

By Gronwall-Bellman's lemma \cite[Pg. 651]{khalil2002nonlinear}, we get
\begin{align}
    |x(t)-w(t)| & \leq \bar{d} e^{\bar{L}t} \leq \bar{d} e^{\bar{L}T},  
\end{align}
where we get the last inequality because we consider the time interval $[0,T]$. Since \eqref{eq:perturb_cond} holds, we obtain
\begin{equation}
    |x(t)-w(t)| \leq \rho^{-1}(\sigma). \label{eq:xw}
\end{equation}

We then obtain \eqref{eq:perturb_results} using (C2'), (V1), \eqref{eq:xw} and $x(0)=w(0)$ using the argument below
\begin{align}
    V(x(t))&=V(w(t))+V(x(t))-V(w(t)) \nonumber \\
    & \leq \beta_{V}(V(w(0)),t) + \rho(|v(t)-w(t)|) \nonumber \\
    & \leq \beta_{V}(V(x(0)),t) + \sigma.
\end{align} \hfill $\Box$